\documentclass[a4paper]{amsart}

\usepackage{amsmath,amscd,color}
\usepackage{amssymb}

\usepackage[top=1.4in, bottom=1.4in, left=1.4in, right=1.4in]{geometry}

\usepackage{amsthm}
\theoremstyle{plain}
\newtheorem{mydef}{Definition}[section]

\newtheorem{mylemma}[mydef]{Lemma}
\newtheorem{mytheorem}[mydef]{Theorem}
\newtheorem{mycorollary}[mydef]{Corollary}
\newtheorem{myprop}[mydef]{Proposition}

\theoremstyle{definition}

\newtheorem{myremark}[mydef]{Remark}
\theoremstyle{remark}

\usepackage{enumitem}

\usepackage{comment}

\usepackage{cancel}

\usepackage{mathtools}

\usepackage{tikz}
\usetikzlibrary{matrix,arrows}

\title[The Role of the Jacobi Identity in Solving the Maurer-Cartan Structure Eqn]{The Role of the Jacobi Identity in Solving the Maurer-Cartan Structure Equation}
\author{Ori Yudilevich}
\date{November 12, 2015}

\begin{document}

\begin{abstract}
 We describe a method for solving the Maurer-Cartan structure equation associated with a Lie algebra that isolates the role of the Jacobi identity as an obstruction to integration. We show that the method naturally adapts to two other interesting situations: local symplectic realizations of Poisson structures, in which case our method sheds light on the role of the Poisson condition as an obstruction to realization; and the Maurer-Cartan structure equation associated with a Lie algebroid, in which case we obtain an explicit formula for a solution to the equation which generalizes the well known formula in the case of Lie algebras. 
 \end{abstract}

\maketitle

\section*{Introduction}

\subsection*{Realization Problem for Lie Algebras}

Any Lie group $G$ carries a canonical 1-form with values in the tangent space to the identity $\mathfrak{g}$, 
\begin{equation*}
 \phi\in\Omega^1(G;\mathfrak{g}),
\end{equation*}
known as the \textit{Maurer-Cartan form} of $G$. Actually, the Lie group structure is encoded, in some sense, in the 1-form and its properties; this is in fact Cartan's approach to Lie's infinitesimal theory. The two main properties of the Maurer-Cartan form are: it satisfies the so-called \textit{Maurer-Cartan structure equation}\footnote{This paper deals with the classical Maurer-Cartan equation. To avoid confusion with the Maurer-Cartan equation that appears in the context of differential graded Lie algebras and other areas, we use the term \textit{Maurer-Cartan Structure Equation}.}{} and it is pointwise an isomorphism (the latter is often phrased as the property that the components of the 1-form with respect to some basis form a coframe). The Maurer-Cartan structure equation reveals a Lie algebra structure on $\mathfrak{g}$. Of course, the resulting Lie algebra is the same one obtained in the more common approach of using invariant vector fields. 

Conversely, if we begin with an $n$-dimensional Lie algebra $\mathfrak{g}$, we can formulate the following problem, known as the \textit{realization problem for Lie algebras}: find a $\mathfrak{g}$-valued 1-form $\phi\in\Omega^1(U;\mathfrak{g})$ defined on some open neighborhood $U\subset\mathfrak{g}$ of the origin such that $\phi$ is pointwise an isomorphism and satisfies the Maurer-Cartan structure equation
\begin{equation}
 d\phi+\frac{1}{2}[\phi,\phi]=0.
 \label{eqn:maurercartan}
\end{equation}
A solution to the problem induces a local Lie group structure on some open subset of $U$ (see \cite{Greub1973}, p. 368-369) and, therefore, we can think of this realization problem as the problem of locally integrating Lie algebras. 

A solution to this problem is obtained by supposing that the Lie algebra integrates to a Lie group, and pulling back the canonical Maurer-Cartan form on the Lie group by the exponential map. This produces the following $\mathfrak{g}$-valued 1-form $\phi\in\Omega^1(\mathfrak{g};\mathfrak{g})$, whose defining formula refers only to data coming from the Lie algebra and not from the Lie group:
\begin{equation}
 \phi_x(y) = \int_0^1 e^{-t\;\text{ad}_x}y\;dt,\;\;\;\;\;\;\;\; x\in\mathfrak{g},\; y\in T_x\mathfrak{g}.
 \label{eqn:solutionintro}
\end{equation}
This formula defines a solution to \eqref{eqn:maurercartan}, as can be verified directly, and since it is equal to the identity at the origin, it is pointwise an isomorphism in a neighborhood of the origin. See \cite{Duistermaat2000,Sternberg2004} for more details. 

We now make the following observation: neither Equation \eqref{eqn:maurercartan} nor Formula \eqref{eqn:solutionintro} rely on the Jacobi identity; they make perfect sense if we replace the Lie algebra with the weaker notion of a \textit{pre-Lie algebra}, namely a vector space $\mathfrak{g}$ equipped with an antisymmetric bilinear map $[\cdot,\cdot]:\mathfrak{g}\times\mathfrak{g}\to\mathfrak{g}$. However, \eqref{eqn:solutionintro} is a solution of \eqref{eqn:maurercartan} if and only if $\mathfrak{g}$ is a Lie algebra, which is not difficult to show. This leads to the natural question: what is the precise role of the Jacobi identity? Put differently, at what point in the integration process does the Jacobi identity appear?

In Section \ref{section:liealgebra} we present a 2-step method for solving the realization problem for Lie algebras which answers this question. The method can be outlined as follows:
\begin{itemize}
 \item Step 1 (Theorem \ref{theorem:uniqueness}): we formulate a weaker version of the realization problem, which admits a unique solution given any pre-Lie algebra. 
 \item Step 2 (Theorem \ref{theorem:jacobi}): we show that the solution of the weak realization problem is a solution of the complete realization problem if and only if the Jacobi identity is satisfied.
\end{itemize}
Two nice features of the method are:
\begin{itemize}
 \item Step 1 produces an explicit formula for a solution,
 \item Step 2 gives an explicit relation between the Maurer-Cartan structure equation and the Jacobi identity. Loosely speaking, one is the derivative of the other. 
\end{itemize}

\subsection*{Similar Phenomenon: Poisson Realizations} There is a striking similarity between the phenomenon we just observed and a phenomenon that occurs in the story of symplectic realizations of Poisson manifolds. Recall that a Poisson manifold $(M,\pi)$ is a manifold $M$ equipped with a bivector $\pi$ which satisfies the Poisson equation $[\pi,\pi]=0$ (of course, the Poisson equation is equivalent to the condition that the induced Poisson bracket satisfy the Jacobi identity). A symplectic realization of a Poisson manifold $(M,\pi)$ is a symplectic manifold $(S,\omega)$ together with a surjective submersion $p:S\to M$ that satisfy the equation 
\begin{equation}
 dp(\omega^{-1}) = \pi.
 \label{eqn:symprealizationequationintro}
\end{equation}
It was shown in \cite{crainic2011-3} that for any Poisson manifold $(M,\pi)$, a symplectic realization is explicitly given by the cotangent bundle $T^*M$ equipped with the symplectic form
\begin{equation}
 \omega = \int_0^1 (\varphi_t)^*\omega_{\text{can}} dt
 \label{eqn:symprealizationsolutionintro}
\end{equation}
together with the projection $p:T^*M\to M$. Here, $\omega_{\text{can}}$ is the canonical symplectic form and $\varphi_t$ is the flow associated with a choice of a contravariant spray on $T^*M$. See \cite{crainic2011-3} for more details.  

As in the realization problem of Lie algebras, we make the following observation: neither Equation \eqref{eqn:symprealizationequationintro} nor Formula \eqref{eqn:symprealizationsolutionintro} depend on the Poisson equation; they make perfect sense when replacing $\pi$ with any bivector. And as before, there is the natural question as to the precise role of the Poisson equation in the existence of symplectic realizations, a question which was raised in \cite{crainic2011-3} (see last paragraph of the paper). 

An explicit relation between the symplectic realization equation and the Maurer-Cartan structure equation was observed by Alan Weinstein \cite{Weinstein1983} in his pioneer work on Poisson manifolds. Weinstein showed that, locally, \eqref{eqn:symprealizationequationintro} is equivalent to a Maurer-Cartan structure equation associated with an infinite dimensional Lie algebra, and exploited this to prove the existence of local symplectic realizations by using a heuristic argument to solve this Maurer-Cartan structure equation, producing an explicit local solution of the type \eqref{eqn:symprealizationsolutionintro}.

In Section \ref{section:algebroid}, we apply our method to solve the Maurer-Cartan structure equation which Weinstein formulated. As with Lie algebras, we do this by identifying a weaker version of the equation that admits a unique solution given any bivector, not necessarily Poisson, and proceed to show that the solution is a local symplectic realization if and only if the bivector satisfies the Poisson equation. We obtain an explicit relation between the Poisson equation and the symplectic realization condition, thus pinpointing the role of the Poisson equation in the problem of existence of local symplectic realizations.  

\subsection*{The Lie Algebroid Case}

In addition to local symplectic realizations, we believe that our method can be adapted to various other situations which generalize or resemble the classical Lie algebra case. One important generalization, which we treat in Section \ref{section:algebroid}, is the realization problem of a Lie algebroid. Although extra difficulties do arise, it is remarkable that the procedure continues to work in this case, despite the fact that the simple to handle bilinear bracket of a Lie algebra is replaced by a more cumbersome bi-differential operator.  This is largely facilitated by the presence of certain flows, known as infinitesimal flows, associated with time-dependent sections of the Lie algebroid. 

As we noted in ``Step 1'' above, our method produces an explicit solution. In the Lie algebra case, this is the well known Formula \eqref{eqn:solutionintro}, whereas the formula we obtain in the Lie algebroid case does not appear in the literature to the best of our knowledge (see Theorem \ref{theorem:uniquenessalgebroid}). Having this explicit formula at hand can prove to be useful; in particular, one can attempt to use it to explicitly integrate Lie algebroids locally (as an indication of feasibility, in \cite{Coste1987} a symplectic realization of a Poisson manifold was used to integrate the associated Lie algebroid to a local symplectic groupoid, see also discussion in subsection \ref{subsection:poissonvsalgebroid}). 

\subsection*{Final Remark}
We would like to end the introduction with a historical remark and to briefly describe our motivation for reopening this classical problem. The Maurer-Cartan structure equation originates in the work of \'Elie Cartan \cite{Cartan1904,Cartan1937-1} under the name of ``Structure Equations''. In his work on Lie pseudogroups, Cartan associates the equation with a Lie pseudogroup, and subsequently extracts out of the equation the Lie pseudogroup's ``structure functions'', i.e. its infinitesimal data. The reverse direction, the problem of finding and classifying the solutions to the structure equations associated with a given infinitesimal data, is known as the realization problem, two special cases of which we discussed above (the Lie algebra case and the Lie algebroid case).

This work arose as part of a larger project aimed at understanding Cartan's original work on Lie pseudogroups in a global, more geometric and coordinate-free fashion, and in particular, the realization problem. Since Cartan's realization problem involves infinitesimal structures that fail to satisfy the Jacobi identity, we first tried to understand the role of the Jacobi identity in the integration process of structures for which the Jacobi identity is satisfied, namely Lie algebras and Lie algebroids. The method and the results that we came across and that we are presenting here seemed to have relevance beyond the realization problem itself, and we, therefore, decided to present it in an independent fashion. 

\subsection*{Acknowledgements}
This research was financially supported by the ERC Starting Grant no. 279729. The author would also like to thank Marius Crainic, Pedro Frejlich, Ionut Marcut and Maria Amelia Salazar for helpful discussions and suggestions, and the anonymous referee for excellent remarks and suggestions.

\section{The Maurer-Cartan structure equation of a Lie algebra}
\label{section:liealgebra}

In this section, we present the 2-step method for solving the realization problem for a Lie algebra which was outlined in the introduction. Let us first recall the necessary definitions.

\begin{mydef}
 A \textbf{pre-Lie algebra} is a vector space $\mathfrak{g}$ equipped with an antisymmetric bilinear map $[\cdot,\cdot]:\mathfrak{g}\times\mathfrak{g}\to\mathfrak{g}$. A \textbf{Lie algebra} is a pre-Lie algebra that satisfies the Jacobi identity:
 \begin{equation*}
  [[x,y],z] + [[y,z],x] + [[z,x],y] = 0,\;\;\;\;\; \forall x,y,z\in\mathfrak{g}.
 \end{equation*}
\end{mydef}
Associated with a pre-Lie algebra is the \textbf{adjoint map} $\text{ad}:\mathfrak{g}\to\mbox{End}(\mathfrak{g}),\; \text{ad}_x(y)=[x,y]$, and the \textbf{Jacobiator}
\begin{equation}
 \mbox{Jac}\in \mbox{Hom}(\Lambda^3\mathfrak{g},\mathfrak{g}),\;\;\;\;\;\mbox{Jac}(x,y,z) = [[x,y],z] + [[y,z],x] + [[z,x],y].
 \label{eqn:jacobiatordef}
\end{equation}

The space of $\mathfrak{g}$-valued differential forms on $\mathfrak{g}$ is denoted by $\Omega^*(\mathfrak{g};\mathfrak{g})$. This space is equipped with the de Rham differential $d:\Omega^*(\mathfrak{g};\mathfrak{g}) \to \Omega^{*+1}(\mathfrak{g};\mathfrak{g})$ and with a bracket, $[\cdot,\cdot]:\Omega^p(\mathfrak{g};\mathfrak{g})\times\Omega^q(\mathfrak{g};\mathfrak{g})\to\Omega^{p+q}(\mathfrak{g};\mathfrak{g})$, that plays the role of the wedge product on $\mathfrak{g}$-valued forms and is defined by the analogous formula:
\begin{equation}
[\omega,\eta](X_1,...,X_{p+q}) = \sum_{\sigma\in S_{p,q}} \text{sgn}(\sigma) [\omega(X_{\sigma(1)},...,X_{\sigma(p)}),\eta(X_{\sigma(p+1)},...,X_{\sigma(p+q)})],
\label{eqn:bracketforms}
\end{equation}
where $S_{p,q}$ is the set of $(p,q)$-shuffles. 

Of course, given any open subset $U\subset \mathfrak{g}$, we also have the space of $\mathfrak{g}$-valued forms $\Omega^*(U;\mathfrak{g})$ on $U$ equipped with a differential and a bracket, defined in the same manner. Given any $\phi\in\Omega^1(U;\mathfrak{g})$, the \textbf{Maurer-Cartan 2-form} associated with $\phi$ is defined by:
\begin{equation*}
 \text{MC}_\phi:=d\phi+\frac{1}{2}[\phi,\phi]\in\Omega^2(U;\mathfrak{g}),
\end{equation*}
and the \textbf{Maurer-Cartan structure equation} is
\begin{equation*}
 \text{MC}_\phi = 0,
\end{equation*}
or more explicitly, 
\begin{equation*}
(\text{MC}_\phi)_x(y,z)=0, \;\;\;\;\; \forall \;x\in U,\;y,z\in\mathfrak{g}.
\end{equation*}
Note that in the last equation, and throughout the paper, we identify the tangent spaces of a vector space with the vector space itself without further mention. 

Recall the \textbf{realization problem for Lie algebras}: find a 1-form $\phi\in\Omega^1(U;\mathfrak{g})$ on some open neighborhood $U\subset \mathfrak{g}$ of the origin such that $\phi$ is pointwise an isomorphism and satisfies the Maurer-Cartan structure equation.

We now present our method for solving this realization problem.

\subsection{Step 1:} We show that a weaker version of the realization problem admits a solution given any pre-Lie algebra. We accomplish this by imposing a boundary condition which transforms the equation into a simple ODE that can be easily solved.

\begin{mytheorem}
Given any pre-Lie algebra $\mathfrak{g}$, the equation
\begin{equation}
 (\text{MC}_\phi)_x(x,y) = 0,\;\;\;\;\;\;\; \forall x,y\in\mathfrak{g} 
 \label{eqn:equation2}
\end{equation}
admits a solution in $\Omega^1(\mathfrak{g};\mathfrak{g})$ which is pointwise an isomorphism at the origin (and thus on some open neighborhood of the origin). Moreover, if we impose the boundary condition
\begin{equation}
 \phi_x(x) = x,\;\;\;\;\; \forall x\in\mathfrak{g},
 \label{eqn:equation1}
\end{equation}
then the solution is unique and is given by the following formula:
\begin{equation}
\phi_x(y)=\int_0^1 e^{-t\text{ad}_x}\;y\;dt.
\label{eqn:solution}
\end{equation}
\label{theorem:uniqueness}
\end{mytheorem}

\begin{myremark}
To get a ``geometric feel'' of the equations, note that \eqref{eqn:equation2} is the restriction of the Maurer-Cartan structure equation to all two dimensional subspaces of $\mathfrak{g}$,
 and \eqref{eqn:equation1} is the condition that $\phi$ restricts to the identity on all one-dimensional subspaces.
\end{myremark}

\begin{proof} First note that \eqref{eqn:equation1} implies that $\phi_0=\text{id}$, and in particular, $\phi$ is pointwise an isomorphism at the origin.

Let $\phi\in\Omega^1(\mathfrak{g};\mathfrak{g})$ be a solution of \eqref{eqn:equation2} and \eqref{eqn:equation1}. We will show that $\phi$ must be of the form given by  \eqref{eqn:solution}, which implies uniqueness. Conversely, as we will explain at the end of the proof, reading the steps in the reverse direction will imply that  \eqref{eqn:solution} is a solution, thus proving existence.

Note that, by linearity, \eqref{eqn:equation2} and \eqref{eqn:equation1} are equivalent to $(\text{MC}_\phi)_{tx} (x,y) = 0$  and $\phi_{tx}(x)=x$ for all $t\in (0,1)$ and $x,y\in\mathfrak{g}$. In particular, by continuity, this implies that
\begin{equation}
\label{eqn:equationsatorigin}
(\text{MC}_\phi)_0 (x,y) = 0 \;\;\;\;\;\text{and}\;\;\;\;\; \phi_0(x)=x \;\;\;\;\; \forall x,y\in\mathfrak{g}.
\end{equation}

Fix $x,y\in\mathfrak{g}$. The solution $\phi$ satisfies 
\begin{equation}
  (d\phi+\frac{1}{2}[\phi,\phi])_{tx}(x,ty)=0,
\label{eqn:MCrestricted}
\end{equation}
for all $t\in (0,1)$. 

To compute $(d\phi)_{tx}(x,ty)$, consider the map $f:(0,1)\times(-\delta,\delta) \to \mathfrak{g},\; f(t,\epsilon)=t(x+\epsilon y)$. 
\begin{equation*}
\begin{split}
 (d\phi)_{tx}(x,ty) &= (f^*d\phi)_{(t,0)}(\frac{\partial}{\partial t},\frac{\partial}{\partial \epsilon}) = (df^*\phi)_{(t,0)}(\frac{\partial}{\partial t},\frac{\partial}{\partial \epsilon}) \\ 
 & = \frac{\partial}{\partial t} \Big( (f^*\phi)(\frac{\partial}{\partial \epsilon})\Big)\Big|_{(t,0)} - \frac{\partial}{\partial \epsilon} \Big( (f^*\phi)(\frac{\partial}{\partial t}) \Big)\Big|_{(t,0)} \\
 & = \frac{\partial}{\partial t} \big( \phi_{t(x+\epsilon y)}(ty) \big)\big|_{(t,0)} - \frac{\partial}{\partial \epsilon} \big(\phi_{t(x+\epsilon y)}(x+\epsilon y)\big)\big|_{(t,0)} \\
 & = \frac{\partial}{\partial t} \big( \phi_{tx}(ty) \big) - y,
 \end{split}
\end{equation*}
where, in the last equality, we have used that  \eqref{eqn:equation1} implies $\phi_{t(x+\epsilon y)}(x+\epsilon y) = x+\epsilon y$.

To compute $(\frac{1}{2}[\phi,\phi])_{tx}(x,ty)$, we use  \eqref{eqn:equation1} again:
\begin{equation*}
 (\frac{1}{2}[\phi,\phi])_{tx}(x,ty) = [\phi_{tx}(x),\phi_{tx}(ty)]= [x,\phi_{tx}(ty)] = \text{ad}_x(\phi_{tx}(ty)). 
\end{equation*}

Thus for a $\phi$ that satisfies  \eqref{eqn:equation1},  \eqref{eqn:MCrestricted} is equivalent to
\begin{equation*}
\frac{\partial}{\partial t} ( \phi_{tx}(ty) )  - y + \text{ad}_x(\phi_{tx}(ty))=0 , 
\end{equation*}
which is equivalent to
\begin{equation}
 \frac{\partial}{\partial t} \big(e^{t\;\text{ad}_x}\; \phi_{tx}(ty) \big) = e^{t\;\text{ad}_x} \; y.
 \label{eqn:equationbeta}
\end{equation}

Integrating from $0$ to $t'$:
\begin{equation}
\phi_{t'x}(t'y) = \int_0^{t'} e^{(t-t')\;\text{ad}_{x}}\;y\;dt = \int_0^1 e^{-t\;\text{ad}_{t'x}}\;(t'y)\;dt
\label{eqn:betaformula}
\end{equation}
Setting $t'=1$ proves that $\phi$ coincides with \eqref{eqn:solution} .

Next, we show that $\phi$ defined by  \eqref{eqn:solution} is a solution. Note that $\phi_x(x) = \int_0^1 e^{-t\; ad_x}\;x\;dt = \int_0^1 x\;dt = x$, and thus  \eqref{eqn:equation1} is satisfied. Equation \eqref{eqn:solution} is equivalent to  \eqref{eqn:betaformula} which is a solution of \eqref{eqn:equationbeta}, and since $\phi$ satisfies \eqref{eqn:equation1}, it is a solution of \eqref{eqn:MCrestricted}. In particular, setting $t=1$ implies that $(\text{MC}_\phi)_x(x,y)=0$, and thus  \eqref{eqn:equation2} is satisfied.
\end{proof}

\subsection{Step 2:} By obtaining explicit equations relating the Maurer-Cartan 2-form with the Jacobiator, we show that the solution obtained in the previous step is a solution of the Maurer-Cartan structure equation if and only if the Jacobiator vanishes. 

\begin{mytheorem}
\label{theorem:jacobi}
Let $\mathfrak{g}$ be a pre-Lie algebra and $\phi\in \Omega^1(\mathfrak{g};\mathfrak{g})$ the solution of \eqref{eqn:equation2} and \eqref{eqn:equation1}. Then 
\begin{equation*}
MC_\phi=0 \;\;\;\;\; \Longleftrightarrow \;\;\;\;\; \mbox{Jac}=0, 
\end{equation*}
or, more precisely, 
\begin{eqnarray}
& \text{Jac}(x,y,z) = - 3 \frac{d}{dt}  (\text{MC}_\phi)_{tx}(y,z)  \big|_{t=0}, \label{eqn:jacobiator} \\
&(\text{MC}_\phi)_{x}(y,z) =  - \int_0^{1} e^{(t-1)\text{ad}_x}\; \text{Jac}(x,\phi_{tx}(ty),\phi_{tx}(tz)) \; dt . \label{eqn:jacobiobstruction}
\end{eqnarray} 
\end{mytheorem}

\begin{proof}
Equations \eqref{eqn:jacobiator} and \eqref{eqn:jacobiobstruction} imply that $MC_\phi=0$ if and only if $\mbox{Jac}=0$. Let us derive these equations. Fix $x,y,z\in\mathfrak{g}$. We will compute
\begin{equation}
 d(\text{MC}_\phi)_{tx}(x,ty,tz),
 \label{eqn:differentialMC}
\end{equation}
with $t \in (0,1)$, in two different ways.

\noindent 1) Consider the map $f:(0,1)\times(-\delta,\delta)^2 \to \mathfrak{g},\; f(t,\epsilon,\epsilon')=t(x+\epsilon y+\epsilon' z)$, then
\begin{equation*}
 \begin{split}
    (d\text{MC}_\phi)_{tx}(x,ty,tz) &= (f^*d\text{MC}_\phi)_{(t,0,0)}(\frac{\partial}{\partial t},\frac{\partial}{\partial \epsilon},\frac{\partial}{\partial \epsilon'}) = (df^*\text{MC}_\phi)_{(t,0,0)}(\frac{\partial}{\partial t},\frac{\partial}{\partial \epsilon},\frac{\partial}{\partial \epsilon'}) \\
    &=  \frac{d}{d t} (\text{MC}_\phi)_{tx}(ty,tz).  \\
 \end{split}
\end{equation*}
In the last equality, terms containing $(\text{MC}_\phi)_{tx}(x,ty)$ and $(\text{MC}_\phi)_{tx}(x,tz)$ vanish by  \eqref{eqn:equation2}.

\noindent 2) On the other hand,
\begin{equation*}
 \begin{split}
  & (d\text{MC}_\phi)_{tx}(x,ty,tz) = (d\frac{1}{2}[\phi,\phi])_{tx}(x,ty,tz) = ([d\phi,\phi])_{tx}(x,ty,tz)\\ 
  & = [(d\phi)_{tx}(x,ty),\phi_{tx}(tz)] + [(d\phi)_{tx}(tz,x),\phi_{tx}(ty)] + [(d\phi)_{tx}(ty,tz),\phi_{tx}(x)]\\
  & = - [[x,\phi_{tx}(ty)],\phi_{tx}(tz)] - [[\phi_{tx}(tz),x],\phi_{tx}(ty)] + [(\text{MC}_\phi)_{tx}(ty,tz) - [\phi_{tx}(ty),\phi_{tx}(tz)],x] \\
  & = -[x,(\text{MC}_\phi)_{tx}(ty,tz)] - \text{Jac}(x,\phi_{tx}(ty),\phi_{tx}(tz)).
 \end{split} 
\end{equation*}
In the fourth equality, we have used  \eqref{eqn:equation2} and \eqref{eqn:equation1}. In particular,  \eqref{eqn:equation1} implies: $(d\phi)_{tx}(x,y) + [\phi_{tx}(x),\phi_{tx}(y)]=0$ and $(d\phi)_{tx}(x,z) + [\phi_{tx}(x),\phi_{tx}(z)]=0$. 

Equation \eqref{eqn:differentialMC} becomes
\begin{equation*}
 \text{Jac}(x,\phi_{tx}(ty),\phi_{tx}(tz)) = - (\frac{d}{d t} + \text{ad}_x) (\text{MC}_\phi)_{tx}(ty,tz),
\end{equation*}
or equivalently,
\begin{equation}
 e^{t\;\text{ad}_x}\;\text{Jac}(x,\phi_{tx}(ty),\phi_{tx}(tz)) = - \frac{d}{d t} \left( e^{t\;\text{ad}_x}\; (\text{MC}_\phi)_{tx}(ty,tz) \right).
 \label{eqn:jacobiatordiff}
 \end{equation}
Integrating from $0$ to $1$ produces  \eqref{eqn:jacobiobstruction}, while multiplying both sides of the equation by $\frac{1}{t^2}$, taking the limit of $t$ to $0$ and using the fact that $(\text{MC}_\phi)_0(y,z)=0$ (see \eqref{eqn:equationsatorigin}) produces \eqref{eqn:jacobiator}.
\end{proof}

\begin{myremark}
 The method we present here was inspired by the method used in \cite{Sternberg2004} (see sections 1.3-1.5) to compute the differential of the exponential map of a Lie group and to derive the Baker-Campbell-Hausdorff formula of a Lie algebra. 
\end{myremark}

\section{The Maurer-Cartan structure equation and Local Symplectic Realizations of Poisson Structures}
\label{section:poisson}

In this section, we apply the method from the previous section to the problem of existence of symplectic realizations of Poisson manifolds. The role of the Poisson equation becomes manifest, in the same way that the role of the Jacobi identity was made manifest in the Lie algebra case. 

\begin{mydef}
 A \textbf{pre-Poisson manifold} $(M,\pi)$ is a manifold $M$ together with a choice of a bivector field $\pi\in\mathfrak{X}^2(M)$. A \textbf{Poisson manifold} $(M,\pi)$ is a pre-Poisson manifold with the extra condition that $\pi$ satisfies the \textbf{Poisson equation} $[\pi,\pi]=0$ (where $[\cdot,\cdot]$ is the Schouten-Nijenhuis bracket).
\end{mydef}

Equivalently, a pre-Poisson manifold is a manifold $M$ equipped with an $\mathbb{R}$-bilinear antisymmetric operation $\{,\}:C^\infty(M)\times C^\infty(M) \to C^\infty(M)$ (called ``the Poisson bracket'')  that satisfies the Leibniz identity, $\{fg,h\} = f\{g,h\} + \{f,h\}g$ for all $f,g,h\in C^\infty(M)$. A bivector $\pi$ induces a bracket by $\{f,g\}(m) = \pi_m(df,dg)$ for all $m\in M,\;f,g\in C^\infty(M)$, and vice versa. The Poisson equation is equivalent to the Jacobi identity, i.e. to the condition $\text{Jac}=0$, where $\text{Jac}$ is the Jacobiator associated with $\{,\}$ (defined as in the previous section). 

By the Leibniz identity, a function $f\in C^\infty(M)$ induces a vector field $X_f\in \mathfrak{X}(M)$, the \textbf{Hamiltonian vector field} associated with $f$, by the condition $X_f(g) = \{f,g\}$ for all $g\in C^\infty(M)$, or equivalently, $X_f(g) = \pi(df,dg)$ for all $g\in C^\infty(M)$.  

Poisson manifolds can be localized, i.e. if $(M,\pi)$ is a Poisson manifold and $U\subset M$ is an open subset, then $(U,\pi|_U)$ is a Poisson manifold.

A \textbf{symplectic realization} of a Poisson manifold $(M,\pi)$ is a symplectic manifold $(S,\omega)$ together with a surjective submersion $p:S\to M$ such that $p$ is a Poisson map, i.e. the bivector $\omega^{-1}$ induced by the symplectic form $\omega$ is $p$-projectable to the bivector $\pi$, that is to say,
\begin{equation*}
 dp(\omega^{-1})=\pi.
\end{equation*}
A \textbf{local symplectic realization} of $(M,\pi)$ around a point $m\in M$ is a symplectic realization of $(U,\pi|_U)$, where $U$ is some open neighborhood of $m$.

In the problem of existence of local symplectic realizations it is enough to consider Poisson manifolds of the type $(\mathcal{O},\pi)$, where 
\begin{equation*}
 \mathcal{O}\subset V
\end{equation*}
is an open subset of a vector space $V$. The following proposition was proven by Alan Weinstein (\cite{Weinstein1983}, section 9). To be more precise, Weinstein proved it for the case that $(\mathcal{O},\pi)$ is a Poisson manifold; however, the arguments do not rely on the Jacobi identity and the proposition also holds for the case that $(\mathcal{O},\pi)$ is a pre-Poisson manifold.

\begin{myprop}
Let $(\mathcal{O},\pi)$ be a pre-Poisson manifold. Let $\phi\in\Omega^1(V^*;C^\infty(\mathcal{O}))$ be defined by
\begin{equation}
\phi_\xi(\zeta)=\int_0^1 (\varphi^{-t}_{X_\xi})^* \zeta\;dt,\;\;\;\;\; \forall \xi,\zeta\in V^*.
\label{eqn:solutionpoissonintro}
\end{equation}
Here $\xi$ and $\zeta$ are interpreted as linear functionals on $V$, $X_\xi$ is the corresponding Hamiltonian vector field and $\varphi_{X_\xi}$ its flow. 

Let $\tilde{\phi}\in\Omega^1(\mathcal{O}\times V^*)$ be the induced 1-form on $\mathcal{O}\times V^*= T^*\mathcal{O}$ defined by $\tilde{\phi}_{(x,\xi)}(y,\zeta) := \phi_\xi(\zeta)(x)$. 

Then, the 2-form $d\tilde{\phi}$ is symplectic on some neighborhood $U\subset \mathcal{O}\times V^*$ of the zero-section and, writing $p:\mathcal{O}\times V^*\to \mathcal{O}$ for the projection,
\begin{equation*}
 p\big|_U:(U,d\tilde{\phi}) \to  (\mathcal{O},\pi)\; \text{ is a symplectic realization} \;\;\;\;\; \Longleftrightarrow \;\;\;\;\; d\phi + \frac{1}{2}\{\phi,\phi\} = 0.
\end{equation*}
\label{prop:symplecticrealization}
\end{myprop}

\begin{myremark}
The 1-form $\phi$ defined by \eqref{eqn:solutionpoissonintro} and the induced 1-form $\tilde{\phi}$ are only well-defined on some open neighborhood of the zero section of $\mathcal{O}\times V^*$, namely on all points $(x,\xi)$ such that $\varphi_{X_\xi}(x)$ is defined up to time $1$. This does not pose a problem, since, in the end, we are only interested in the symplectic form $d\tilde{\phi}$ in some neighborhood of the zero section. 
\end{myremark}

Weinstein's remarkable observation was that the symplectic realization condition can be locally rephrased as a Maurer-Cartan structure equation. This equation lives in the space $\Omega^*(V^*;C^\infty(\mathcal{O}))$ consisting of differential forms with values in $C^\infty(\mathcal{O})$, where a 1-form $\phi\in\Omega^1(V^*;C^\infty(\mathcal{O}))$ is smooth if the map $\mathcal{O}\times V^* \to \mathbb{R},\; (x,\xi)\mapsto \phi_\xi(\zeta)(x)$, is smooth for all $\zeta\in V^*$, and similarly for higher degree forms. This space is equipped with the de Rham differential $d$ defined as usual, and a bracket $\{,\}$ defined as in \eqref{eqn:bracketforms} (with the Lie bracket replaced by the Poisson bracket); thus, one can make sense of the Maurer-Cartan 2-form associated with a 1-form $\phi\in\Omega^1(V^*;C^\infty(\mathcal{O}))$:
\begin{equation*}
 \text{MC}_\phi:=d\phi+\frac{1}{2}\{\phi,\phi\}\in\Omega^2(V^*;C^\infty(\mathcal{O})).
\end{equation*}

Weinstein proceeded to show that if $(\mathcal{O},\pi)$ is a Poisson manifold, then the 1-form given by \eqref{eqn:solutionpoissonintro} satisfies the Maurer-Cartan structure equation, thus proving the existence of local symplectic realizations. Of course, the fact that the Poisson bracket satisfies the Jacobi identity is used in the proof, but its precise role is somewhat obscure, appearing as a ``mere step'' in the calculation (see \cite{Weinstein1983}, p. 547).

The following two theorems shed further light on the role of the Jacobi identity as an obstruction in this problem. The first of the two theorems, an analog of ``Step 1'' of the previous section, demonstrates how close $d\tilde{\phi}$ induced by \eqref{eqn:solutionpoissonintro} is from being a symplectic realization, regardless of the Jacobi identity. 

\begin{mytheorem}
Let $(\mathcal{O},\pi)$ be a pre-Poisson manifold. The 1-form $\phi\in \Omega^1(V^*;C^\infty(\mathcal{O}))$ defined by \eqref{eqn:solutionpoissonintro} satisfies the equation
\begin{equation}
 (\text{MC}_\phi)_\xi(\xi,\zeta) = 0,\;\;\;\;\;\;\; \forall \xi,\zeta\in V^*. 
 \label{eqn:equation2poisson}
\end{equation}
Moreover, it is the unique solution of \eqref{eqn:equation2poisson} together with the boundary condition
\begin{equation}
 \phi_\xi(\xi) = \xi,\;\;\;\;\; \forall \xi\in V^*.
 \label{eqn:equation1poisson}
\end{equation}
\label{theorem:uniquenesspoisson}
\end{mytheorem}

\begin{proof}
The proof is essentially the same as the proof of Theorem \ref{theorem:uniqueness}. One must only make the following adjustments: 
\begin{itemize}
 \item $\mathfrak{g}$ with $V^*$ (and acccordingly $x,y$ with $\xi,\zeta$),
 \item the Lie bracket $[,]$ with the Poisson bracket $\{,\}$,
 \item $e^{t\;\text{ad}_{\xi}}$ with $(\varphi^t_{X_\xi})^*$,
\end{itemize}
and while making the last of the three adjustments, one notes that derivatives of matrix valued functions of $t$ become derivatives of flows. 
\end{proof}

The next theorem, an analog of ``Step 2'' of the previous section, gives an explicit relation between $\text{Jac}$ and $\text{MC}_\phi$ which translates into a precise relation between the failure of the Poisson equation and the failure of $d\tilde{\phi}$ from being a symplectic realization. Of course, it follows that if the Poisson equation is satisfied, then $d\tilde{\phi}$ is a symplectic realization.

\begin{mytheorem}
\label{theorem:jacobipoisson}
Let $\phi\in\Omega^1(V^*;C^\infty(\mathcal{O}))$ be a solution to \eqref{eqn:equation2poisson} and \eqref{eqn:equation1poisson}, then 
\begin{equation*}
 \text{Jac}=0 \;\;\; \Longleftrightarrow \;\;\; MC_\phi=0,
\end{equation*}
or more precisely,
\begin{equation*}
\begin{split}
&\text{Jac}(\xi,\zeta,\eta) = - 3 \frac{d}{dt}  (\text{MC}_\phi)_{t\xi}(\zeta,\eta) \big|_{t=0}, \\
&(\text{MC}_\phi)_{\xi}(\zeta,\eta) =  - \int_0^{1} (\varphi^{t-1}_{X_\xi})^*\; \text{Jac}(\xi,\phi_{t\xi}(t\zeta),\phi_{t\xi}(t\eta)) \; dt . 
\end{split}
\end{equation*} 
\end{mytheorem}

\begin{proof}
The proof is essentially the same as the proof of Theorem \ref{theorem:jacobi} after making the necessary adjustments as in the proof of the previous theorem, and using the fact that by the Leibniz identity, the vanishing of the Jacobiator on linear functions implies that it vanishes. 
\end{proof}

\begin{myremark}
\label{remark:linearpoisson}
 Theorems \ref{theorem:uniqueness} and \ref{theorem:jacobi} are in fact special cases of theorems \ref{theorem:uniquenesspoisson} and \ref{theorem:jacobipoisson}. Recall that a linear Poisson structure on the vector space $\mathfrak{g}^*$ is a Poisson bracket on $C^\infty(\mathfrak{g}^*)$ satisfying the property that it restricts to a Lie bracket on the linear functions $\mathfrak{g}\subset C^\infty(\mathfrak{g}^*)$. This defines a one-to-one correspondence between linear Poisson structures on $\mathfrak{g}^*$ and Lie algebra structures on $\mathfrak{g}$. In the case of linear Poisson structures, the Hamiltonian vector field on $\mathfrak{g}^*$ associated with an element $x\in\mathfrak{g}=(\mathfrak{g}^*)^*$ is simply the transpose $(\text{ad}_x)^*$ of the linear map $\text{ad}_x:\mathfrak{g}\to\mathfrak{g}$. The flow of $(\text{ad}_x)^*$ is the transpose of the linear map $e^{t\;\text{ad}_x}$, and the pullback by the flow is precisely $e^{t\;\text{ad}_x}$. This implies that the solution \eqref{eqn:solutionpoissonintro} takes values in $\mathfrak{g}\subset C^\infty(\mathfrak{g}^*)$, and it follows that theorems \ref{theorem:uniquenesspoisson} and \ref{theorem:jacobipoisson} for linear Poisson structures coincide with theorems \ref{theorem:uniqueness} and \ref{theorem:jacobi}. 
\end{myremark}

\section{The Maurer-Cartan structure equation of a Lie algebroid}
\label{section:algebroid}

In this section, we generalize our method from the Lie algebra case to the Lie algebroid case. We will begin by recalling the basic definitions and discussing the realization problem for Lie algebroids, after which we will state and prove theorems \ref{theorem:uniquenessalgebroid} and \ref{theorem:jacobialgebroid} which generalize theorems \ref{theorem:uniqueness} and \ref{theorem:jacobi}.

\begin{mydef}
 A \textbf{pre-Lie algebroid} $A\xrightarrow{\pi} M$ is a vector bundle $A$ over $M$ equipped with a vector bundle map (the `anchor') $\rho:A\to TM$  and an antisymmetric bilinear map (the `bracket') $[\cdot,\cdot]:\Gamma(A)\times\Gamma(A)\to\Gamma(A)$ satisfying,
\begin{equation*}
 \begin{split}
  & [\alpha,f\beta] = f[\alpha,\beta] + \mathcal{L}_{\rho(\alpha)}(f) \beta \;\;\;\;\;\; \forall \alpha,\beta\in\Gamma(A),\; f\in C^\infty(M),\\
  & \rho([\alpha,\beta]) = [\rho(\alpha),\rho(\beta)] ,\;\;\;\;\;\;\;\;\;\;\;\;\;\;\; \forall \alpha,\beta\in\Gamma(A).
 \end{split}
\end{equation*}
A pre-Lie algebroid $A\to M$ is called a \textbf{Lie algebroid} if it further satisfies the Jacobi identity
\begin{equation*}
 [[\alpha,\beta],\gamma] + [[\beta,\gamma],\alpha] + [[\gamma,\alpha],\beta] = 0,\;\;\;\;\; \forall \alpha,\beta,\gamma\in\Gamma(A).
\end{equation*}
\end{mydef}

Associated with a pre-Lie algebroid is the \textbf{Jacobiator} tensor $\text{Jac}\in \text{Hom}(\Lambda^3 A,A)$, defined at the level of sections by
\begin{equation*}
 \text{Jac}(\alpha,\beta,\gamma) = [[\alpha,\beta],\gamma] + [[\beta,\gamma],\alpha] + [[\gamma,\alpha],\beta], \;\;\;\;\; \forall \alpha,\beta,\gamma\in \Gamma(A),
\end{equation*}
and easily checked to be $C^\infty(M)$-linear in all slots. 

The notions of $A$-connections, $A$-paths, geodesics and infinitesimal flows that appear in the context of Lie algebroids remain unchanged when we give up on the Jacobi identity and pass to pre-Lie algebroids. We will assume familiarity with these notions, and otherwise refer the reader to the appendix (and to \cite{crainic2003-1} for more details). 

Let $A\to M$ be a pre-Lie algebroid equipped with an $A$-connection $\overline{\nabla}$. To every point $a\in A$ we associate the unique maximal geodesic $g_a:I_a\to A$ that satisfies $g_a(0)=a$. We denote its base curve by $\gamma_a:I_a\to M$. Let $A_0\subset A$ be a neighborhood of the zero-section such that $g_a$ is defined up to at least time $1$ for all $a\in A_0$. On $A_0$ we have the exponential map $\text{exp}: A_0 \to A,\; a\mapsto g_a(1)$, and the target map $\tau = \pi\circ\text{exp} : A_0\to M$. Let $\Omega^*_\pi(A_0;\tau^*A)$ be the space of foliated differential forms (foliated with respect to the foliation by $\pi$-fibers) with values in $\tau^*A$. Throughout this section we will use the canonical identification between the vertical bundle of $A_0$ and the pullback of $A$ to $A_0$, i.e. $T_aA_0\cong A_x$ for all $a\in (A_0)_x$. Thus, given a 1-form $\phi\in \Omega^1_{\pi}(A_0;\tau^*A)$, we will write $\phi_a(b)$ with $a\in (A_0)_x,\;b\in A_x$. 

A 1-form $\phi\in \Omega^1_{\pi}(A_0;\tau^*A)$ is said to be \textbf{anchored} if $\rho\circ\phi=d\tau$. Given a vector bundle connection $\nabla:\mathfrak{X}(M)\times \Gamma(A)\to \Gamma(A)$, we define the \textbf{Maurer-Cartan 2-form} associated with an anchored 1-form $\phi\in\Omega^1_{\pi}(A_0;\tau^*A)$ to be
\begin{equation*}
 \text{MC}_\phi := d_{\tau^*\nabla}\phi + \frac{1}{2}[\phi,\phi]_\nabla\in\Omega^2_\pi(A_0;\tau^*A).
\end{equation*}
The differential-like map $d_{\tau^*\nabla}$ and bracket on $\Omega^*_\pi(A_0;\tau^*A)$ are defined in the usual way (see appendix). The anchored condition implies that $\text{MC}_\phi$ is independent of the choice of connection (Proposition \ref{prop:MCindependence}). The auxiliary connection $\nabla$ should not be confused with the $A$-connection $\overline{\nabla}$, which is part of the data we fix.

Of course, given any open subset $U\subset A_0$, we have the space of forms $\Omega^*_\pi(U;\tau^*A)$ equipped with a differential-like operator and a bracket in the same manner, and anchored 1-forms have associated Maurer-Cartan 2-forms. The \textbf{realization problem for Lie algebroids} can now be stated: find an anchored 1-form $\phi\in\Omega^1_\pi(U;\tau^*A)$ on some open neighborhood of the zero-section of $A_0$ such that $\phi$ is pointwise an isomorphism and satisfies the \textbf{Maurer-Cartan structure equation}:
\begin{equation}
 \text{MC}_\phi = 0.
 \label{eqn:maurercartanalgebroid}
\end{equation}

\begin{myremark}
	A solution of the Maurer-Cartan structure equation can also be interpreted as a Lie algebroid map: a 1-form $\phi\in\Omega^1_\pi(A_0;\tau^*A)$ can be viewed as a vector bundle map from the Lie algebroid $T^\pi A_0\to A_0$ (the vertical bundle, a Lie subalgebroid of $TA_0\to A_0$) to the Lie algebroid $A\to M$ covering $\tau$, the anchored condition on $\phi$ is equivalent to the vector bundle map commuting with the anchors, and $\phi$ satisfies the Maurer-Cartan structure equation if and only if the vector bundle map is a Lie algebroid map (see \cite{Crainic2011-1} or \cite{Fernandes2014} for more details). From this point of view, the Maurer-Cartan structure equation is a special case of the \textit{generalized Maurer-Cartan equation} for vector bundle maps between Lie algebroids which commute with the anchors studied in \cite{Fernandes2014} (section 3.2).
\end{myremark}

As in the case of Lie algebras (see introduction), one can find a solution to the realization problem by assuming that the Lie algebroid integrates to a Lie groupoid and pulling back the canonical Maurer-Cartan 1-form on the Lie groupoid by the exponential map. The resulting formula will not depend on the Lie groupoid, and one can verify directly that the formula is indeed a solution, and, therefore, not have to require that the Lie algebroid be integrable.   

Let us explain this in more detail. Let $\mathcal{G} \rightrightarrows M$ be a Lie groupoid with source/target map $s/t$. The canonical Maurer-Cartan 1-form $\phi_{\text{MC}}\in\Omega^1_s(\mathcal{G};t^*A)$ is a foliated differential 1-form on $\mathcal{G}$ (foliated with respect to the foliation by $s$-fibers) with values in $t^*A$. It is defined precisely as in the case of Lie groups:
\begin{equation}
 (\phi_{\text{MC}})_g = (dR_{g^{-1}})_g,\;\;\;\; \forall g\in \mathcal{G},
 \label{eqn:groupoidmcform}
\end{equation}
the difference being that the right multiplication map $R_{g^{-1}}$ is only defined on $s^{-1}(s(g))$. For this reason, the resulting form is foliated. The Maurer-Cartan form satisfies the anchored property $\rho((\phi_{\text{MC}})_g(X)) = (dt)_g(X)$ and the Maurer-Cartan structure equation $d_{t^*\nabla}\phi_{\text{MC}} + \frac{1}{2}[\phi_{\text{MC}},\phi_{\text{MC}}]_\nabla = 0$ (for more details, see \cite{Fernandes2014}, section 4).

The exponential map $\text{Exp}:=\text{Exp}_{\overline{\nabla}}:A_0\to\mathcal{G}$ on a Lie groupoid requires a choice of an $A$-connection $\overline{\nabla}$ on $A$, where $A_0$ is as above. Such a choice induces a normal connection on each $s$-fiber and the exponential map is then defined in the usual way. This choice of an $A$-connection also gives rise to an exponential on the Lie algebroid, as we saw above, and the two satisfy the following relations:
\begin{gather}
\text{exp}(a)=(dR_{\text{Exp}(a)^{-1}})_{\text{Exp}(a)}\frac{d}{dt} \text{Exp}(ta)\big|_{t=1}, \label{eqn:Expexprelation1} \\
\pi\circ\text{exp} = t\circ \text{Exp}, \label{eqn:Expexprelation2} \\
\pi = s\circ \text{Exp}. \notag
\end{gather}
If we pull back the Maurer-Cartan form by the exponential map, the resulting form will be an element of $\Omega^1_\pi(A_0;\tau^*A)$. It will be anchored as a result of \eqref{eqn:Expexprelation2}. It is now not difficult to verify that the fact that $\phi_{\text{MC}}$ satisfies the Maurer-Cartan structure equation on the Lie groupoid implies that $\text{Exp}^*\phi_{\text{MC}}$ satisfies the Maurer-Cartan structure equation on the Lie algebroid, i.e. satisfies \eqref{eqn:maurercartanalgebroid}.

In the following two theorems we will obtain a solution by taking a different path, namely by generalizing our method from section \ref{section:liealgebra}. The first theorem is the generalization of ``Step 1'': a weaker version of the realization problem which admits a unique solution for any pre-Lie algebroid. The theorem gives an explicit formula for a solution to the realization problem of Lie algebroids. In Corollary \ref{cor:groupoidmcformformula} we show that our solution coincides with $\text{Exp}^*\phi_{\text{MC}}$.

\begin{mytheorem}
Let $A\to M$ be a pre-Lie algebroid equipped with an $A$-connection $\overline{\nabla}$. The equations
\begin{align}
 & (\text{MC}_\phi)_a(a,b) = 0,\;\;\;\;\;\;\;\;\;\;\;\;\;\;\; \forall x\in M,\; a\in (A_0)_x,\;b\in A_x,  \label{eqn:equation2algebroid} \\
 & \rho \circ \phi = d\tau, \label{eqn:equation3algebroid}
\end{align}
admit a solution in $\Omega^1_\pi(A_0;\tau^*A)$ which is pointwise an isomorphism on a small enough neighborhood of the zero section of $A_0$. Moreover, if we impose the boundary condition
\begin{equation}
 \phi_a(a) = \text{exp}(a),\;\;\;\;\;\forall a\in A_0, 
 \label{eqn:equation1algebroid} 
\end{equation}
then the solution is unique and can be described as follows: let $\xi:[0,1]\times(-\delta,\delta)\times M \to A$ be a smooth map such that $\xi^t_\epsilon = \xi (t,\epsilon,\cdot)$ is a section of $A$ and $\xi_\epsilon^t(\gamma_{a+\epsilon b}(t)) = g_{a+\epsilon b}(t)$ for all $(t,\epsilon)\in [0,1]\times(-\delta,\delta)$, and let $\psi_{\xi_0}$ be the infinitesimal flow associated with the time dependent section $\xi_0$ (see appendix). The solution is given by
\begin{equation}
\phi_a(b)=\int_0^1 \psi^{1,t}_{\xi_0} \frac{d}{d\epsilon}\Big|_{\epsilon=0} \xi_\epsilon^t(\gamma_a(t))  \;dt.
\label{eqn:solutionalgebroid}
\end{equation}
\label{theorem:uniquenessalgebroid}
\end{mytheorem}
\begin{proof}
Equation \eqref{eqn:equation1algebroid} implies that a solution $\phi$ is equal to the identity on the zero section of $A$ and thus pointwise an isomorphism on a small enough neighborhood of the zero section.

Let $\phi\in \Omega^1_\pi(A_0;\tau^* A)$ be a solution of \eqref{eqn:equation2algebroid}, \eqref{eqn:equation3algebroid} and \eqref{eqn:equation1algebroid}. In this proof we show that $\phi$ must be given by  \eqref{eqn:solutionalgebroid}. The remaining arguments are precisely as in the proof of Theorem \ref{theorem:uniqueness}.

By  \eqref{eqn:equation1algebroid}, $\phi_a(a) = \text{exp}(a) = g_a(1)$ for all $a\in A_0$. This implies that $\phi_{ta}(ta) = g_{ta}(1) = tg_a(t)$, by using  \eqref{eqn:geodesiclinearity}, and by linearity,
\begin{equation}
 \phi_{ta}(a) = g_a(t)
 \label{eqn:equation2algebroid_consequence}
\end{equation}
for all $t\in (0,1)$. Equation \eqref{eqn:equation2algebroid_consequence} is thus equivalent to  \eqref{eqn:equation1algebroid}.

Fix $x\in M,\; a\in (A_0)_x$ and $b\in A_x$. Let $\nabla$ be a vector bundle connection on $A$. Equation \eqref{eqn:equation2algebroid} implies that
\begin{equation}
 (d_{\tau^*\nabla}\phi + \frac{1}{2}[\phi,\phi]_\nabla)_{ta}(a,tb) = 0,
 \label{eqn:MCrestrictedalgebroid}
\end{equation}
for all $t\in (0,1)$. We will compute this equation for a fixed $t'\in (0,1)$. 

To compute $(d_{\tau^*\nabla}\phi)_{t'a}(a,t'b)$, consider the map $f:(0,1)\times (-\delta,\delta) \to (A_0)_x,\; f(t,\epsilon) = t(a+\epsilon b)$. The composition $\tau \circ f$ restricted to $\epsilon=0$ is the curve $t\mapsto \tau(ta)$ which is precisely $\gamma_a$, the base curve of the geodesic $g_a$, and $\tau \circ f$ restricted to $t=t'$ is the curve $\gamma_\epsilon:(-\delta,\delta)\to M,\; \epsilon\mapsto \tau(t'(a+\epsilon b))$.   
\begin{equation*}
 \begin{split}
   (d_{\tau^*\nabla}\phi)_{t'a}(a,t'b) &= (f^*d_{\tau^*\nabla}\phi)_{(t',0)}( \frac{\partial}{\partial t},\frac{\partial}{\partial \epsilon}) = (d_{f^*\tau^*\nabla} f^*\phi)_{(t',0)}( \frac{\partial}{\partial t},\frac{\partial}{\partial \epsilon}) \\
   & = (f^*\tau^*\nabla)_{\frac{\partial}{\partial t}} (f^*\phi)(\frac{\partial}{\partial \epsilon})\Big|_{(t',0)} - (f^*\tau^*\nabla)_{\frac{\partial}{\partial \epsilon}} (f^*\phi)(\frac{\partial}{\partial t})\Big|_{(t',0)} \\
   & = (\nabla)_{\dot{\gamma}_a} \phi_{ta} (tb)\Big|_{t=t'} - (\nabla)_{\dot{\gamma}_\epsilon} g_{a+\epsilon b}(t')\Big|_{\epsilon=0}
 \end{split}
\end{equation*}
In the second equality we have used Lemma \ref{lemma:pullbackdifferential} to commute the pullback with $d_{\tau^*\nabla}$ and in the last equality we have used  \eqref{eqn:equation2algebroid_consequence} which is equivalent to  \eqref{eqn:equation1algebroid}. The two terms in the final expression are covariant derivatives of paths, which make sense because $\gamma_a$ is the base curve of the curve $t \mapsto \phi_{ta}(tb)$ and $\gamma_\epsilon$ is the base curve of $\epsilon\mapsto g_{a+\epsilon b}(t')$.

To compute $(\frac{1}{2}[\phi,\phi]_\nabla)_{t'a}(a,t'b)$, let $\xi$ be the map as in the statement of the theorem and let $\eta$ be a time dependent section of $A$ satisfying $\eta^t(\gamma_a(t)) = \phi_{ta}(tb)$. 
\begin{equation*}
 \begin{split}
   (\frac{1}{2}[\phi,\phi]_\nabla)_{t'a}(a,t'b) &= [ \xi^{t'}_0,\eta^{t'}]_\nabla(\gamma_a(t')) \\ 
   &= [ \xi^{t'}_0,\eta^{t'}](\gamma_a(t')) - \nabla_{\rho(\xi_0^{t'})} \eta^{t'} (\gamma_a(t')) + \nabla_{\rho(\eta^{t'})} \xi_0^{t'} (\gamma_a(t')) \\
   &= \frac{d}{dt}\Big|_{t=t'}\psi^{t',t}_{\xi_0} \eta^{t'}(\gamma_a(t)) - \nabla_{\dot{\gamma}_a} \eta^{t'} (\gamma_a(t')) + \nabla_{\dot{\gamma}_\epsilon} \xi_0^{t'} (\gamma_a(t'))
 \end{split}
\end{equation*}
In the last equality, we have used the defining property \eqref{eqn:infinitesimalflow} of the infinitesimal flow for the first term, $\rho(\xi_0^{t'}(\gamma_a(t'))) = \rho(g_a(t')) = \dot{\gamma}_a(t')$ for the second term, and 
\begin{equation*}
\begin{split}
 \rho(\eta^{t'}(\gamma_a(t'))) &= \rho(\phi_{t'a}(t'b)) = (d\tau)_{t'a}(t'b) = d(\pi\circ\text{exp})_{t'a}(t'b) = \frac{d}{d\epsilon}\Big|_{\epsilon=0}(\pi(exp(t'a+\epsilon t' b))) \\ 
 &= \frac{d}{d\epsilon}\Big|_{\epsilon=0}(\pi(g_{t'(a+\epsilon b)}(1))) = \frac{d}{d\epsilon}\Big|_{\epsilon=0}(\pi(t'g_{(a+\epsilon b)}(t'))) = \dot{\gamma}_\epsilon(0)
 \end{split}
\end{equation*}
for the third term, where we have used the anchored property \eqref{eqn:equation3algebroid} in the second equality.

Thus for $\phi$ that satisfies  \eqref{eqn:equation1algebroid},  \eqref{eqn:MCrestrictedalgebroid} is equivalent to
\begin{equation*}
 \frac{d}{dt}\Big|_{t=t'}\psi^{t',t}_{\xi_0} \eta^{t'}(\gamma_a(t)) + \frac{d}{dt}\Big|_{t=t'}\eta^t(\gamma_a(t'))  = \frac{d}{d\epsilon}\Big|_{\epsilon=0}\xi^{t'}_\epsilon(\gamma_a(t')),
\end{equation*}
where we have used the characterization \eqref{eqn:covariantderivative} of covariant derivatives of curves.

Applying $\psi^{1,t'}_{\xi_0}$ to both sides and using the product rule, the latter equation is equivalent to
\begin{equation*}
 \frac{d}{dt}\psi^{1,t}_{\xi_0} \eta^{t}(\gamma_a(t)) = \psi^{1,t}_{\xi_0} \frac{d}{d\epsilon}\Big|_{\epsilon=0}\xi^{t}_\epsilon(\gamma_a(t)).
\end{equation*}
Integrating $t'$ from $0$ to $1$, and using the definition of $\eta$ and the property $\psi^{1,1}_{\xi_0}=\text{id}$, we obtain  \eqref{eqn:solutionalgebroid}.
\end{proof}

\begin{mycorollary}
\label{cor:groupoidmcformformula}
The pullback of the canonical Maurer-Cartan form of a Lie groupoid by the exponential map $\text{Exp}^*\phi_{\text{MC}}$ is equal to the 1-form defined by \eqref{eqn:solutionalgebroid}.
\end{mycorollary}

\begin{proof}
	We saw already in the text preceding the last theorem that the 1-form $\text{Exp}^*\phi_{\text{MC}}\in\Omega^1_\pi(A_0;\tau^*A)$ is anchored and satisfies the Maurer-Cartan structure equation, and, in particular, it satisfies \eqref{eqn:equation2algebroid}. Moreover, the initial condition \eqref{eqn:equation1algebroid} is satisfied since it is precisely the relation \eqref{eqn:Expexprelation1} when written out explicitly. The corollary now follows from the uniqueness assertion in the theorem.
\end{proof}

The second theorem is the generalization of ``Step 2'' from section \ref{section:liealgebra}. It shows that the solution from the previous theorem is indeed a solution of the realization problem.

\begin{mytheorem}
\label{theorem:jacobialgebroid}
Let $A$ be a pre-Lie algebroid and $\phi\in \Omega^1_\pi(A_0;\tau^*A)$ a solution of \eqref{eqn:equation2algebroid}, \eqref{eqn:equation3algebroid} and \eqref{eqn:equation1algebroid}. Choose $A_0$ to be small enough so that $\phi$ is pointwise an isomorphism. Then $MC_\phi=0$ if and only if $\mbox{Jac}=0$, or more precisely,
\begin{eqnarray}
& \text{Jac}(a,b,c) = - 3 \frac{d}{dt} \left(  \psi^{0,t}_\xi (\text{MC}_\phi)_{ta}(b,c) \right) \big|_{t=0}, \label{eqn:jacobiatoralgebroid} \\
&(\text{MC}_\phi)_{a}(b,c) =  - \int_0^{1} \psi^{1,t}_\xi \text{Jac}(\frac{1}{t}\text{exp}(ta),\phi_{ta}(tb),\phi_{ta}(tc)) \; dt, \label{eqn:jacobiobstructionalgebroid}
\end{eqnarray} 
where $\xi$ is a time dependent section of $A$ satisfying $\xi^t(\gamma_a(t)) = g_a(t)$ for all $t\in (0,1)$.
\end{mytheorem}

\begin{proof}
The proof goes along the same line as the proof of Theorem \ref{theorem:jacobi}. As in Theorem \ref{theorem:jacobi}, we will compute
\begin{equation}
 d_{\tau^*\nabla}(\text{MC}_\phi)_{ta}(a,tb,tc)
 \label{eqn:differentialMCalgebroid}
\end{equation}
in two different way, where $t \in (0,1),\;x\in M,\;a\in (A_0)_x, \;y,z\in A_x$ and $\nabla$ some vector bundle connection on $A$.

\noindent 1) Consider the map $f:(0,1)\times(-\delta,\delta)^2 \to \mathfrak{g},\; f(t,\epsilon,\epsilon')=t(a+\epsilon b+\epsilon' c)$. Recall that $\gamma_a$ is the base curve of the geodesic $g_a$ that satisfies $\gamma_a(t)=\tau(ta)$.
\begin{equation*}
 \begin{split}
    (d_{\tau^*\nabla}\text{MC}_\phi)_{ta}(a,tb,tc) &= (f^*d_{\tau^*\nabla}\text{MC}_\phi)_{(t,0,0)}(\frac{\partial}{\partial t},\frac{\partial}{\partial \epsilon},\frac{\partial}{\partial \epsilon'}) \\ &= (d_{f^*\tau^*\nabla}f^*\text{MC}_\phi)_{(t,0,0)}(\frac{\partial}{\partial t},\frac{\partial}{\partial \epsilon},\frac{\partial}{\partial \epsilon'}) \\
    &=  \nabla_{\dot{\gamma}_a} (\text{MC}_\phi)_{ta}(tb,tc),  \\
 \end{split}
\end{equation*}
where the final expression is the covariant derivative of the curve $t\mapsto (\text{MC}_\phi)_{ta}(tb,tc)$ covering $\gamma_a$. 

\noindent 2) Since $\phi\in\Omega^1_\pi(A_0;\tau^*A)$ is a pointwise isomorphism, it induces a linear map $\phi^{-1}:\Gamma(A)\to \mathfrak{X}(A_0)$. Let $\xi$ be as in the statement of the theorem, let $\eta_b$ and $\eta_c$ be time dependent sections of $A$ satisfying $\eta_b^t(\gamma_a(t)) = \phi_{ta}(tb)$ and $\eta_c^t(\gamma_a(t)) = \phi_{ta}(tc)$ and let $\sigma$ be a time dependent section of $A$ satisfying $\sigma^t(\gamma_a(t)) = (\text{MC}_\phi)_{ta}(tb,tc)$. Let $\tilde{a},\tilde{b},\tilde{c}$ be time dependent vector fields on $A_0$ defined by $\tilde{a}^t = \phi^{-1} (\xi^t),\;\tilde{b}^t = \phi^{-1} (\eta_b^t),\;\tilde{c}^t = \phi^{-1} (\eta_c^t)$. 
\begin{equation*}
 \begin{split}
  & (d_{\tau^*\nabla}\text{MC}_\phi)_{ta}(a,tb,tc) = (d_{\tau^*\nabla}\text{MC}_\phi)(\tilde{a},t\tilde{b},t\tilde{c})_{ta}  \\
  & = \nabla_{\dot{\gamma}_a}\sigma^t (\gamma_a(t)) - [[\xi^t,\eta^t_b],\eta^t_c]_{ta} - [[\eta^t_c,\xi^t],\eta^t_b]_{ta} + [\sigma^t - [\eta^t_b,\eta^t_c],\xi^t]_{ta} \\
  & = \nabla_{\dot{\gamma}_a}\sigma^t (\gamma_a(t)) - \frac{d}{ds}\Big|_{s=t} \psi^{t,s}_\xi \sigma^t(\gamma_a(s)) - \text{Jac}(\frac{1}{t}\text{exp}(ta),\phi_{ta}(tb),\phi_{ta}(tc)).
 \end{split} 
\end{equation*}
The second equality is a slightly messy yet straightforward computation. It involves expanding $\text{MC}_\phi$ with respect to the chosen connection, using the choices we made above of time dependent sections, and using \eqref{eqn:equation2algebroid}, \eqref{eqn:equation1algebroid} and \eqref{eqn:equation3algebroid}. In particular, it is used that  \eqref{eqn:equation2algebroid} implies that: $\phi_{ta}([\tilde{a},\tilde{b}]) = [\xi^t,\eta_b^t]_{\gamma_a(t)},\;\phi_{ta}([\tilde{a},\tilde{c}]) = [\xi^t,\eta_c^t]_{\gamma_a(t)}$. Furthermore, $(\text{MC}_\phi)_{ta}(b,c) = - \phi_{ta}([\tilde{b},\tilde{c}]) + [\eta_b^t,\eta_c^t]_{\gamma_a(t)}$. In the last equality we express the bracket $[\xi^t,\sigma^t]$ using the infinitesimal flow, see  \eqref{eqn:infinitesimalflow}.

After equating the two expressions obtained, using characterization \eqref{eqn:covariantderivative} of covariant derivatives of curves and applying $\psi^{1,t}_\xi$,  \eqref{eqn:differentialMCalgebroid} becomes
\begin{equation}
 \psi^{1,t}_\xi\text{Jac}(\frac{1}{t}\text{exp}(ta),\phi_{ta}(tb),\phi_{ta}(tc)) = - \frac{d}{d t} \left( \psi^{1,t}_\xi (\text{MC}_\phi)_{ta}(tb,tc) \right).
 \label{eqn:jacobiatordiffalgebroid}
 \end{equation}
The remaining arguments are identical to Theorem \ref{theorem:jacobi}.
\end{proof}

\subsection{The Poisson Case vs. the Lie Algebroid Case}
\label{subsection:poissonvsalgebroid}
Given the well known relations between Poisson manifolds and Lie algebroids, it is natural to wonder as to the relation between the instances of the Maurer-Cartan structure equation associated with these structures, i.e. as to the relation between Section \ref{section:poisson} and Section \ref{section:algebroid} of this paper. Let us briefly touch upon this. 

In one direction, any Lie algebroid $A\to M$ induces a Poisson structure on the total space of the dual vector bundle $A^*\to M$ known as a linear Poisson structure (see \cite{Mackenzie2005}). This generalizes the construction of a linear Poisson structure on the dual of a Lie algebra. At the level of the associated Maurer-Cartan structure equations, it is not hard to verify that, locally and under obvious identifications, the Maurer-Cartan structure equations as well as the solutions are one and the same on both sides of this correspondence. In particular, trivializing $A$ and computing the 1-form \eqref{eqn:solutionalgebroid} will produce the same result as obtained by computing the 1-form \eqref{eqn:solutionpoissonintro} associated with the induced trivialization of $A^*$. This is, of course, a generalization of the case of a Lie algebra which was discussed in Remark \ref{remark:linearpoisson}.

In the opposite direction, any Poisson manifold $(M,\pi)$ induces a Lie algebroid structure on the cotangent bundle $T^*M\to M$, as originally shown in \cite{Coste1987}. In that same paper, the authors prove that the local symplectic realization constructed by Weinstein in \cite{Weinstein1983} (and discussed in section \ref{section:poisson} above) has a canonically induced local symplectic groupoid structure on its total space whose associated Lie algebroid is (the restriction of) $T^*M\to M$. This same phenomenon occurs at the level of the Maurer-Cartan structure equations. Using the notation of Section \ref{section:poisson}, the local solution of the Maurer-Cartan structure equation associated with the Poisson manifold $(\mathcal{O},\pi)$, with $\mathcal{O}\subset V$, induces a local solution to the Maurer-Cartan structure equation associated with the Lie algebroid $T^*\mathcal{O} = \mathcal{O}\times V^* \xrightarrow{\pi} \mathcal{O}$ by differentiation of the coefficients, or more precisely, by the map
\begin{equation}
\label{eqn:relationpoissonalgebroid}
\Omega^1(V^*;C^\infty(\mathcal{O}))\to \Omega^1_\pi((T^*\mathcal{O})_0;\tau^*(T^*\mathcal{O})),\;\;\; \phi \mapsto \hat{\phi},
\end{equation}
with $\hat{\phi}_{x,\xi}(\zeta) = d(\phi_\xi(\zeta))_{\tau(x)}$ for all $x\in \mathcal{O},\;\xi,\zeta\in V^*$.

Note that whereas in the Lie algebroid case we are able to obtain a ``wide'' solution, i.e. on an open neighborhood of the zero section of $T^*M\to M$, in the Poisson case we only obtain a local one around a point in $M$. It would be interesting to further investigate the relation given by \eqref{eqn:relationpoissonalgebroid} to see if a ``wide'' solution of the Lie algebroid case induces a ``wide'' solution of the Poisson case, thus producing yet another proof for the existence of global symplectic realizations. 

\appendix

\section{Facts on (pre-)Lie Algebroids}

In this appendix, various notions are recalled which are needed in section \ref{section:algebroid} for the formulation of the Maurer-Cartan structure equation on a Lie algebroid and its solution. For more details, the reader is referred to \cite{crainic2003-1}. Note that all the notions that appear here and that are presented in \cite{crainic2003-1} do not require the Jacobi identity and are therefore valid for pre-Lie algebroids as they are for Lie algebroids.

Let $A\to M$ be a pre-Lie algebroid (see section \ref{section:algebroid} for the definition). An \textbf{$A$-connection} on a vector bundle $E\to M$ is an $\mathbb{R}$-bilinear map $\overline{\nabla}:\Gamma(A)\times \Gamma(E) \to \Gamma(E)$ satisfying the connection-like properties
\begin{equation*}
\overline{\nabla}_{f\alpha}s = f\overline{\nabla}_\alpha s,\;\;\;\; \overline{\nabla}_\alpha(fs) = f\overline{\nabla}_\alpha s + \mathcal{L}_{\rho(\alpha)}(f)s,\;\;\;\;\forall \alpha\in\Gamma(A),\;s\in\Gamma(E),\; f\in C^\infty(M).
\end{equation*}

For the remainder of the appendix, let $A\to M$ be a pre-Lie algebroid equipped with an $A$-connection $\overline{\nabla}$. Note that there will be two different connections that will play a role in this appendix (and in section \ref{section:algebroid}): an $A$-connection $\overline{\nabla}$ on $A$ that is part of the data, and an auxiliary vector bundle connection $\nabla$ on $A$ that is used to write down the Maurer-Cartan structure equation globally, and which is not part of the data. 

\subsection{Time Dependent Sections}

A \textbf{time dependent section} $\xi$ of $A$ is a map $\xi:I\times M\to A,\; (t,x)\mapsto \xi^t(x)$ (with $I$ some open interval), such that $\xi^t$ is a section of $A$ for all $t\in I$. 

If $\nabla:\mathfrak{X}(M)\times \Gamma(A)\to \Gamma(A)$ is a vector bundle connection, then given a base curve $\gamma:I\to M$ and a curve $u:I\to A$ covering $\gamma$, the covariant derivative $(\nabla_{\dot{\gamma}}u)(t) = ((\gamma^*\nabla)_{\frac{\partial}{\partial t}} u)(t)$ can be characterized using time dependent sections as follows: choose a time dependent section $\xi$ of $A$ satisfying $\xi^t(\gamma(t)) = u(t)$ for all $t\in I$, then 
 \begin{equation}
  (\nabla_{\dot{\gamma}}u)(t) = (\nabla_{\dot{\gamma}}\xi^t)(x) + \frac{d\xi^t}{dt}(x),
  \label{eqn:covariantderivative}
 \end{equation}
where $x=\gamma(t)$. 

We will also use time dependent sections to express the bracket of a pre-Lie algebroid in a Lie derivative-like fashion, as one does for the bracket of vector fields. This involves the notiַַon of an infinitesimal flow. Let $\xi$ be a time dependent section of $A$ and $\rho(\xi)$ the corresponding time dependent vector field on $M$. Let $\varphi^{t,s}_{\rho(\xi)}$ denote the flow of $\rho(\xi)$ from time $s$ to $t$. The \textbf{infinitesimal flow},
\begin{equation*}
 \psi^{t,s}_\xi: A_x \to A_{\varphi^{t,s}_{\rho(\xi)}},\;\;\;\;\; x\in M,
\end{equation*}
is the unique linear map satisfying the properties $\psi^{u,t}_\xi\circ\psi^{t,s}_\xi = \psi^{u,s}_\xi$, $\psi^{s,s}_\xi=\text{id}$ and
\begin{equation*}
 \frac{d}{dt}\Big|_{t=s} \psi^{s,t}_\xi \alpha(\varphi^{t,s}_{\rho(\xi)}(x)) = [\xi^s,\alpha]_x,\;\;\;\;\;\forall \alpha\in\Gamma(A),\; x\in M.
\end{equation*}
Defining the pullback of sections by the infinitesimal flow as $(\psi^{t,s}_\xi)^*(\alpha)(x) = \psi^{s,t}_\xi \alpha(\varphi^{t,s}_{\rho(\xi)}(x))$ for all $\alpha\in\Gamma(A),\; x\in M$, the previous equation can be expressed in the more familiar form
\begin{equation}
 \frac{d}{dt}\Big|_{t=s} (\psi^{t,s}_\xi)^* \alpha = [\xi^s,\alpha],\;\;\;\;\;\forall \alpha\in\Gamma(A).
 \label{eqn:infinitesimalflow}
\end{equation}
For more on infinitesimal flows and their global counterparts, flows along invariant time dependent vector fields on Lie groupoids, see \cite{crainic2003-1}.

\subsection{Geodesics}
 
An \textbf{$A$-path} is a curve $g:I\to A$ with base curve $\gamma:I\to M, \gamma(t)=\pi(g(t))$, such that
\begin{equation*}
 \rho(g(t)) = \dot{\gamma}(t),\;\;\;\;\forall t\in I.
\end{equation*}

Let $g$ be an $A$-path with base curve $\gamma$, and let $u: I\to A$  be another curve covering $\gamma$. The \textbf{covariant derivative} of $u$ with respect to $g$ is the curve $\overline{\nabla}_gu : I \to A$, which is defined in analogy to the usual covariant derivative described above: choose a time dependent section $\xi$ of $A$ satisfying $\xi^t(\gamma(t)) = u(t)$ for all $t\in I$, then 
 \begin{equation*}
  (\overline{\nabla}_gu)(t) = (\overline{\nabla}_g\xi^t)(x) + \frac{d\xi^t}{dt}(x),
 \end{equation*}
where $x=\gamma(t)$.

A \textbf{geodesic} is a curve $g:I\to A$ satisfying the geodesic equation $\overline{\nabla}_gg=0$. Geodesics are $A$-paths. Given any point $a\in A$, there is a unique maximal geodesic $g_a:I_a\to A$ satisfying $g_a(0)=a$ with domain $I_a$. The base curve of $g_a$ will be denoted by $\gamma_a$. Geodesics satisfy the following basic property:
\begin{equation}
   g_{sa}(t) = sg_a(st),\;\;\;\;\; \forall\; a\in A,\;s,t\in\mathbb{R},\;t\in I_{sa},
   \label{eqn:geodesiclinearity}
 \end{equation} 
which can be easily verified by checking that the curve $t\mapsto sg_a(st)$ satisfies the geodesic equation and then by noting that by uniqueness it must be equal to $g_{sa}$ since at time $0$ it takes the value $sa$.

Let $A_0 \subset A$ be a neighborhood of the zero section such that $g_a$ is defined up to time $1$ for all $a\in A_0$.  The \textbf{exponential map} is defined as $\text{exp}:A_0\to A,\; a \mapsto  g_a(1)$. The point $\pi(\text{exp}(a))\in M$ will be called the \textbf{target} of $a$ and $\tau = \pi\circ\text{exp}:A_0\to M$ the \textbf{target map}. 

\subsection{The Maurer-Cartan 2-Form}

Let $\Omega^*_\pi(A_0;\tau^*A)$ denote the space of foliated differential forms on $A_0$ (foliated with respect to the foliation by $\pi$-fibers) which take values in $\tau^*A$. 

Let $\nabla:\mathfrak{X}(M)\times\Gamma(A)\to\Gamma(A)$ be a vector bundle connection on $A$. The torsion $[\cdot,\cdot]_\nabla \in \text{Hom}(\Lambda^2A,A)$ of $\nabla$ is defined at the level of sections as follows
\begin{equation*}
 [\alpha,\beta]_\nabla = [\alpha,\beta] - \nabla_{\rho(\alpha)}\beta + \nabla_{\rho(\beta)}\alpha,\;\;\;\;\; \forall \alpha,\beta\in \Gamma(A),
\end{equation*}
and easily checked to be $C^\infty(M)$-linear in both slots. The torsion induces a bracket $[\cdot,\cdot]_\nabla:\Omega^p_\pi(A_0;\tau^*A)
\times\Omega^q_\pi(A_0;\tau^*A)\to\Omega^{p+q}_\pi(A_0;\tau^*A)$ which plays the role of the wedge product on $A$-valued forms, and similarly to the wedge product, it is defined by the following formula
\begin{equation*}
[\omega,\eta]_\nabla(X_1,...,X_{p+q})_a = \sum_{\sigma\in S_{p,q}} \text{sgn}(\sigma) [\omega(X_{\sigma(1)},...,X_{\sigma(p)})_a,\eta(X_{\sigma(p+1)},...,X_{\sigma(p+q)})_a]_\nabla,
\end{equation*}
for all $a\in A_0$, where $S_{p,q}$ is the set of $(p,q)$-shuffles.

In general, a connection $\nabla$ on a vector bundle $E\to M$ induces a differential-like map $d_\nabla:\Omega^*(M;E)\to \Omega^{*+1}(M;E)$ by the usual Koszul-type formula. For example, if $\phi\in\Omega^1(M;E)$,
\begin{equation*}
 d_\nabla\phi(X,Y) = \nabla_X\phi(Y) - \nabla_Y\phi(X) - \phi([X,Y]),\;\;\;\;\;\forall X,Y\in\mathfrak{X}(M).
\end{equation*}
The map $d_\nabla$ squares to zero if and only if the connection is flat. If $M$ has a foliation $\mathcal{F}$ and $\Omega^*_\mathcal{F}(M;E)$ are the foliated forms, then the map $d_\nabla$ descends to a map of foliated forms $d_\nabla:\Omega^*_\mathcal{F}(M;E)\to \Omega^{*+1}_\mathcal{F}(M;E)$. We will need the following property whose proof is elementary and will be left out:

\begin{mylemma}
 Let $E\to M$ be a vector bundle equipped with a connection $\nabla$ and let $f: N \hookrightarrow M$ be a submanifold. Then the following property holds:
 \begin{equation*}
    f^* d_\nabla \phi = d_{f^*\nabla} f^*\phi,
 \end{equation*}
for any $\phi\in \Omega^*(M;E)$. If $N$ and $M$ are foliated and $f$ is a foliated map, then the property holds for $\phi\in\Omega^*_\mathcal{F}(M;E)$.
\label{lemma:pullbackdifferential}
\end{mylemma}
In our particular case, the induced pull-back connection $\tau^*\nabla$ on the vector bundle $\tau^*A\to A_0$ induces a differential-like map $d_{\tau^*\nabla}:\Omega^*_\pi(A_0;\tau^*A) \to \Omega^{*+1}_\pi(A_0;\tau^*A)$.

A 1-form $\phi\in\Omega^1_\pi(A_0;\tau^*A)$ is said to be \textbf{anchored} if $\rho\circ \phi = d\tau$, or more explicitly, if $\rho(\phi_a(b))=(d\tau)_a(b)$ for all $a\in (A_0)_x,\;b\in A_x$ (where we are using the canonical identification $T_aA_0 \cong A_x$).

\begin{myprop}
 Let $\phi\in\Omega^1_\pi(A_0;\tau^*A)$. If $\phi$ is anchored, then the 2-form
 \begin{equation}
  d_{\tau^*\nabla}\phi + \frac{1}{2}[\phi,\phi]_\nabla \in \Omega^2_\pi(A_0;\tau^*A)
  \label{eqn:maurercartantwoform}
 \end{equation}
 is independent of the choice of connection $\nabla$.
 \label{prop:MCindependence}
\end{myprop}

\begin{proof}
Let $\nabla$ and $\nabla'$ be two connections, then by the defining properties of a connection, $\omega:=\nabla-\nabla' \in \Omega^1(M;\text{Hom}(A,A))$. Let $a\in A_0$ and $X,Y\in T_aA_0$ s.t. $d\pi(X) = d\pi(Y)=0$. On the one hand,
\begin{equation*}
 (d_{\tau^*\nabla}\phi - d_{\tau^*\nabla'}\phi)(X,Y) = \omega_{\tau(a)}((d\tau)_a(X))(\phi_a(Y)) - \omega_{\tau(a)}((d\tau)_a(Y))(\phi_a(X)),
\end{equation*}
on the other hand,
\begin{equation*}
 (\frac{1}{2}[\phi,\phi]_\nabla - \frac{1}{2}[\phi,\phi]_{\nabla'})(X,Y) = -\omega_{\tau(a)}(\rho(\phi_a(X)))(\phi_a(Y)) + \omega_{\tau(a)}(\rho(\phi_a(Y)))(\phi_a(X)).
\end{equation*}
The sum of these two equations vanishes if $\phi$ is anchored.
\end{proof}

We call the 2-form given by \eqref{eqn:maurercartantwoform} the \textbf{Maurer-Cartan 2-form} and denote it by $\text{MC}_\phi$. 

\bibliographystyle{plain}
\bibliography{../../../bibliography/bibliography}

\begin{thebibliography}{10}

\bibitem{Cartan1904}
{\'E}lie Cartan.
\newblock Sur la structure des groupes infinis de transformation.
\newblock {\em Ann. Sci. \'Ecole Norm. Sup. (3)}, 21:153--206, 1904.

\bibitem{Cartan1937-1}
{\'E}lie Cartan.
\newblock La structure des groupes infinis.
\newblock {\em S\'eminaire de Math.}, 4e ann{\'e}e, 1936-1937.

\bibitem{Coste1987}
A.~Coste, P.~Dazord, and A.~Weinstein.
\newblock Groupo\"\i des symplectiques.
\newblock In {\em Publications du {D}\'epartement de {M}ath\'ematiques.
  {N}ouvelle {S}\'erie. {A}, {V}ol.\ 2}, volume~87 of {\em Publ. D\'ep. Math.
  Nouvelle S\'er. A}, pages i--ii, 1--62. Univ. Claude-Bernard, Lyon, 1987.

\bibitem{crainic2003-1}
Marius Crainic and Rui~Loja Fernandes.
\newblock Integrability of {L}ie brackets.
\newblock {\em Ann. of Math. (2)}, 157(2):575--620, 2003.

\bibitem{Crainic2011-1}
Marius Crainic and Rui~Loja Fernandes.
\newblock Lectures on integrability of {L}ie brackets.
\newblock In {\em Lectures on {P}oisson geometry}, volume~17 of {\em Geom.
  Topol. Monogr.}, pages 1--107. Geom. Topol. Publ., Coventry, 2011.

\bibitem{crainic2011-3}
Marius Crainic and Ioan M{\u{a}}rcu{\c{t}}.
\newblock On the existence of symplectic realizations.
\newblock {\em J. Symplectic Geom.}, 9(4):435--444, 2011.

\bibitem{Duistermaat2000}
J.~J. Duistermaat and J.~A.~C. Kolk.
\newblock {\em Lie groups}.
\newblock Universitext. Springer-Verlag, Berlin, 2000.

\bibitem{Fernandes2014}
Rui~Loja Fernandes and Ivan Struchiner.
\newblock The classifying {L}ie algebroid of a geometric structure {I}:
  {C}lasses of coframes.
\newblock {\em Trans. Amer. Math. Soc.}, 366(5):2419--2462, 2014.

\bibitem{Greub1973}
Werner Greub, Stephen Halperin, and Ray Vanstone.
\newblock {\em Connections, curvature, and cohomology. {V}ol. {II}: {L}ie
  groups, principal bundles, and characteristic classes}.
\newblock Academic Press [A subsidiary of Harcourt Brace Jovanovich,
  Publishers], New York-London, 1973.
\newblock Pure and Applied Mathematics, Vol. 47-II.

\bibitem{Mackenzie2005}
Kirill C.~H. Mackenzie.
\newblock {\em General theory of {L}ie groupoids and {L}ie algebroids}, volume
  213 of {\em London Mathematical Society Lecture Note Series}.
\newblock Cambridge University Press, Cambridge, 2005.

\bibitem{Sternberg2004}
S.~Sternberg.
\newblock Lie algebras (lecture notes), April 2004.

\bibitem{Weinstein1983}
Alan Weinstein.
\newblock The local structure of {P}oisson manifolds.
\newblock {\em J. Differential Geom.}, 18(3):523--557, 1983.

\end{thebibliography}

\end{document}